\documentclass[11pt,twoside, final]{amsart}
\copyrightinfo{0}{Iranian Mathematical Society}
\pagespan{1}{\pageref*{LastPage}}
\usepackage{etoolbox,lastpage}
\commby{}
\usepackage{amsmath,amsthm,amscd,amsfonts,amssymb,enumerate}
\usepackage{graphicx}		
\usepackage{color}
\usepackage[colorlinks]{hyperref}

\newtheorem{theorem}{Theorem}[section]
\newtheorem{proposition}[theorem]{Proposition}
\newtheorem{lemma}[theorem]{Lemma}
\newtheorem{corollary}[theorem]{Corollary}
\theoremstyle{definition}
\newtheorem{definition}[theorem]{Definition}

\theoremstyle{remark}
\newtheorem{remark}[theorem]{Remark}
\numberwithin{equation}{section}
\def \reg{\mathrm{reg}}
\def \pd{\mathrm{pd}}
\def \set{\mathrm{set}}
 \begin{document}


\title[On the facet ideal of an expanded simplicial complex]{On the facet ideal of an expanded simplicial complex}

\author[Somayeh Moradi]{Somayeh Moradi}
\address[Somayeh Moradi]{Department of Mathematics, Ilam University, P.O.Box 69315-516,
Ilam, Iran and School of Mathematics, Institute
for Research in Fundamental Sciences (IPM), P.O.Box 19395-5746, Tehran, Iran}
\email{somayeh.moradi1@gmail.com}

\author[Rahim Rahmati-Asghar]{Rahim Rahmati-Asghar}
\address[Rahim Rahmati-Asghar]{Department of Mathematics, Faculty of Basic Sciences, University of Maragheh,
P. O. Box 55181-83111, Maragheh, Iran and
School of Mathematics, Institute for Research in Fundamental Sciences (IPM),
P.O.Box 19395-5746, Tehran, Iran. }
\email{rahmatiasghar.r@gmail.com}

%

 \maketitle
%

\begin{abstract}
For a simplicial complex $\Delta$, the affect of the expansion functor on combinatorial properties of $\Delta$ and algebraic properties of its Stanley-Reisner ring has been studied in some previous papers.
In this paper, we consider the facet ideal $I(\Delta)$ and its Alexander dual which we denote by $J_{\Delta}$ to see how the expansion functor alter the algebraic properties of these ideals. It is shown that for any expansion $\Delta^{\alpha}$ the ideals $J_{\Delta}$ and $J_{\Delta^{\alpha}}$ have the same total Betti numbers and their Cohen-Macaulayness are equivalent, which implies that the regularities of the ideals $I(\Delta)$ and $I(\Delta^{\alpha})$ are equal. Moreover, the projective dimensions of $I(\Delta)$ and $I(\Delta^{\alpha})$ are compared.
In the sequel for a graph $G$, some properties that are equivalent in $G$ and its expansions are presented and for a Cohen-Macaulay (resp. sequentially Cohen-Macaulay and shellable) graph $G$, we give some conditions for adding or removing a vertex from $G$,  so that the remaining graph is still Cohen-Macaulay (resp. sequentially Cohen-Macaulay and shellable).\\
\end{abstract}

\section{\bf Introduction}

Making modifications to a simplicial complex or a monomial ideal so that they fulfill some special properties is a tool to construct new objects with some desired properties and has been considered in many research papers, see for example \cite{BaHe,BiVa,CoNa,FH,KhMo}.
In \cite{KhMo} the expansion functor on a simplicial complex was defined and some algebraic and combinatorial properties of a simplicial complex and its expansions were compared. Generalizing the results in \cite{KhMo}, the authors, in \cite{RaMo} studied the Stanley-Reisner ideal of a simplicial complex and that of its expansions  and it was proved that properties like being Cohen-Macaulay, sequentially Cohen-Macaulay, Buchsbaum and k-decomposable for these ideals are equivalent.

Since any squarefree monomial $I$ ideal may also be considered as the facet ideal of a simplicial complex, this natural question arises that how does the expansion functor on a simplicial complex affect algebraic
properties of the facet ideal of it.
In this paper we consider the facet ideal of a simplicial complex $\Delta$, and its Alexander dual $J_{\Delta}$ via the expansion functor on $\Delta$.

The paper proceeds as follows. In the first section, we recall some preliminaries which are needed in the sequel.
Section 2 is devoted to the study of the facet ideal of an expanded complex and its Alexander dual. One of the main results is the following theorem.

\begin{proposition} (Proposition \ref{facet})
Let $\Delta$ be a simplicial complex, $\alpha\in \mathbb{N}^n$ and $J_{\Delta}$ denotes the Alexander dual of the facet ideal $I(\Delta)$. Then
\begin{enumerate}[\upshape (i)]
  \item $\beta_i(S/J_\Delta)=\beta_i(S^\alpha/J_{\Delta^\alpha})$;
  \item $S/J_\Delta$ is Cohen-Macaulay if and only if $S^\alpha/J_{\Delta^\alpha}$ is Cohen-Macaulay.
\end{enumerate}
\end{proposition}

This implies that $\reg(I(\Delta))=\reg(I(\Delta^{\alpha}))$. Moreover, $I(\Delta)$ has a linear resolution if and only if $I(\Delta^\alpha)$ has a linear resolution.
In Proposition \ref{pdlinear}, we prove that if the facet ideal $I(\Delta)$ has linear quotients and $\alpha=(s,s,\ldots,s)$,
then $\pd(I(\Delta^\alpha))\leq \pd(I(\Delta))s+(d+1)(s-1)$, where $d=\dim(\Delta)$. Moreover, if $\Delta$ is pure, then $$\pd(I(\Delta^\alpha))=\pd(I(\Delta))s+(d+1)(s-1).$$
In Section 3, we consider the case that $\Delta=G$ is a graph. In Theorem \ref{graph}, we find some properties that are equivalent in $G$ and its expansions. We show that for a Cohen-Macaulay (resp. sequentially Cohen-Macaulay and shellable) graph $G$ and a vertex $x\in V(G)$, if we add a new vertex $x'$ to $G$  and connect it to $x$ and all of its neighbours, then the new graph is again Cohen-Macaulay (resp. sequentially Cohen-Macaulay and shellable). Also we give a condition so that by removing a vertex $x$ from $G$, the graph $G\setminus x$ is still  Cohen-Macaulay (resp. sequentially Cohen-Macaulay and shellable) (see Corollaries \ref{G1}, \ref{G2}).

\section{Preliminaries}

Throughout this paper, we assume that $\Delta$ is a simplicial complex on the vertex set $X=\{x_1, \dots, x_n\}$,
$K$ is a field and  $S=K[X]$ is a polynomial ring.
The set of facets (maximal faces) of $\Delta$
is denoted by $\mathcal{F}(\Delta)$ and if $\mathcal{F}(\Delta)=\{F_1,\ldots,F_r\}$, we write $\Delta=\langle F_1,\ldots,F_r\rangle$. For a monomial ideal $I$ of $S$, the set of minimal generators of $I$ is denoted by $\mathcal{G}(I)$.
For $\alpha=(s_1,\ldots,s_n)\in \mathbb{N}^n$, we set $X^{\alpha}=\{x_{11},\ldots,x_{1s_1},\ldots,x_{n1},\ldots,x_{ns_n}\}$ and $S^\alpha=K[X^{\alpha}].$

The concept of expansion of a simplicial complex was defined in \cite{KhMo} as follows.

\begin{definition}
Let $\Delta$ be a simplicial complex on $X$, $\alpha=(s_1,\ldots,s_n)\in \mathbb{N}^n$ and $F=\{x_{i_1},\ldots ,x_{i_r}\}$  be a facet of $\Delta$. The \textbf{expansion} of the simplex $\langle F\rangle$ with respect to $\alpha$ is denoted by $\langle F\rangle^\alpha$ and is defined as a simplicial complex on the vertex set $\{x_{i_lt_l}:1\leq l\leq r,\ 1\leq t_l\leq s_{i_l}\}$ with facets $$\{\{x_{i_1j_1},\ldots, x_{i_rj_r}\}: 1\leq j_m\leq s_{i_m}\}.$$
The expansion of $\Delta$ with respect to $\alpha$ is defined as $$\Delta^\alpha=\bigcup_{F\in\Delta}\langle F\rangle^\alpha.$$
A simplicial complex obtained by an expansion, is called an expanded complex.
\end{definition}

\begin{definition}\label{1.2}
{\rm
A monomial ideal $I$ in the ring $S$ has \textbf{linear quotients} if there exists an ordering $f_1, \dots, f_m$ on the minimal generators of $I$ such that the colon ideal $(f_1,\ldots,f_{i-1}):(f_i)$ is generated by a subset of $\{x_1,\ldots,x_n\}$ for all $2\leq i\leq m$. We show this ordering by $f_1<\dots <f_m$ and we call it an order of linear quotients on $\mathcal{G}(I)$.

Let $I$ be a monomial ideal which has linear quotients and $f_1<\dots <f_m$ be an order of linear quotients on the minimal generators of $I$. For any $1\leq i\leq m$, $\set_I(f_i)$ is defined as
$$\set_I(f_i)=\{x_k:\ x_k\in (f_1,\ldots, f_{i-1}) : (f_i)\}.$$
}
\end{definition}

For a $\mathbb{Z}$-graded $S$-module $M$, the \textbf{Castelnuovo-Mumford regularity} (or briefly regularity)
of $M$ is defined as
$$\reg(M) = \max\{j-i: \ \beta_{i,j}(M)\neq 0\},$$
and the \textbf{projective dimension} of $M$ is defined as
$$\pd(M) = \max\{i:\ \beta_{i,j}(M)\neq 0 \ \text{for some}\ j\},$$
where $\beta_{i,j}(M)$ is the $(i,j)$th graded Betti number of $M$.

For a simplicial complex $\Delta$ with the vertex set $X$, the \textbf{Alexander dual simplicial complex} associated to $\Delta$ is defined as
$$\Delta^{\vee}=:\{X\setminus F:\ F\notin \Delta\}.$$

For a squarefree monomial ideal $I=( x_{11}\cdots
x_{1n_1},\ldots,x_{t1}\cdots x_{tn_t})$, the \textbf{Alexander dual ideal} of $I$, denoted by
$I^{\vee}$, is defined as
$$I^{\vee}:=(x_{11},\ldots, x_{1n_1})\cap \cdots \cap (x_{t1},\ldots, x_{tn_t}).$$
For a subset $C\subseteq X$, by $x^C$ we mean the monomial $\prod_{x\in C} x$.
One can see that
$$(I_{\Delta})^{\vee}=(x^{F^c} \ : \ F\in \mathcal{F}(\Delta)), $$
where $I_{\Delta}$ is the Stanley-Reisner ideal associated to $\Delta$ and $F^c=X\setminus F$.
Moreover, $(I_{\Delta})^{\vee}=I_{\Delta^{\vee}}$.

A simplicial complex $\Delta$ is called Cohen-Macaulay (resp. sequentially Cohen-Macaulay, Buchsbaum and Gorenstein), if its the Stanley Reisner ring $K[\Delta]=S/I_{\Delta}$ is Cohen-Macaulay (resp. sequentially Cohen-Macaulay, Buchsbaum and Gorenstein).
For a graph $G$, with the vertex set $V(G)$ and the edge set $E(G)$, the \textbf{independence complex} of $G$ is defined as $$\Delta_G=\{F\subseteq V(G):\ e\nsubseteq F,\ \forall e\in E(G)\}.$$
The graph $G$  is called Cohen-Macaulay (resp. sequentially Cohen-Macaulay and shellable) if $\Delta_G$ is Cohen-Macaulay (resp. sequentially Cohen-Macaulay and shellable).

For a simplicial complex $\Delta$, the facet ideal of $\Delta$ is defined as $I(\Delta)=(x^F:\ F\in \mathcal{F}(\Delta))$.
Also the complement of $\Delta$ is the simplicial complex $\Delta^c=\langle F^c:\ F\in \mathcal{F}(\Delta)\rangle$. In fact $I(\Delta^c)=I_{\Delta^{\vee}}$.

\section{On the facet ideal of an expanded complex}

This section is devoted to the study of facet ideal of an expanded complex and its Alexander dual to see how their algebraic properties change via the expansion functor.
Set $J_\Delta:=I_{\Delta^c}$. Then one can see that $(J_\Delta)^\vee=I(\Delta)$ and $J_\Delta=\bigcap_{F\in \mathcal{F}(\Delta)}P_F$, where $P_F=(x_i:\ x_i\in F)$.

For any $1\leq i\leq n$, let $\delta_i=(a_1,\ldots,a_n)\in \mathbb{N}^n$, where the components are defined as
$a_j=\left\{
\begin{array}{ll}
0  & \hbox{if}\ j\neq i \\
1 & \hbox{if}\ j=i.
\end{array}
\right. $
Also let $\mathbf{1}=(1,1,\ldots,1)\in \mathbb{N}^n$.

\begin{lemma}\label{epsilon}
Let $\Delta$ be a simplicial complex on $X$, $\beta=(k_1,\ldots,k_n)\in \mathbb{N}^n$ and $\alpha=\beta+\delta_i$. Then $(\Delta^\beta)^{\mathbf{1}+\delta_{ik_i}}\cong \Delta^\alpha$.
\end{lemma}
\begin{proof}
Note that
$$X^\alpha=\{x_{11},\ldots,x_{1k_1},\ldots,x_{i1},\ldots,x_{ik_i},x_{i(k_i+1)},\ldots,x_{n1},\ldots,x_{nk_n}\}$$
and
$$(X^\beta)^{\mathbf{1}+\delta_{ik_i}}=\{x_{111},\ldots,x_{1k_11},\ldots,x_{i11},\ldots,x_{ik_i1},x_{ik_i2},\ldots,x_{n11},\ldots,x_{nk_n1}\}.$$

Define $\varphi:(X^\beta)^{\mathbf{1}+\delta_{ik_i}}\rightarrow X^\alpha$ given by
$$\varphi(x_{rst})=\left\{
  \begin{array}{ll}
    x_{rs} & t=1 \\
    x_{i(k_i+1)} & t=2.
  \end{array}
\right.$$

Then $\varphi$ induces the simplicial map

$$\begin{array}{cr}
    \theta: & (\Delta^\beta)^{\mathbf{1}+\delta_{ik_i}}\rightarrow \Delta^\alpha. \\
     & F\mapsto \varphi(F)
  \end{array}$$
which is an isomorphism.
\end{proof}

The following lemma explains the generators of $J_{\Delta^\alpha}$ in terms of the generators of $J_{\Delta}$.
\begin{lemma}\label{J}
Let $\Delta$ be a simplicial complex and let $\alpha=(k_1,\ldots,k_n)\in \mathbb{N}^n$. Then
$$J_{\Delta^\alpha}=(\{\prod^{k_{i_1}}_{j=1}x_{i_1j}\cdots \prod^{k_{i_r}}_{j=1}x_{i_rj}:x_{i_1}\cdots x_{i_r}\in J_{\Delta}\}).$$
\end{lemma}
\begin{proof}
It suffices to show $``\subseteq$''. We use induction on $\alpha$.

Let $\alpha=\mathbf{1}+\delta_1$. When $\Delta=\langle \{x_1,\ldots, x_n\}\rangle$ then $J_{\Delta^\alpha}=\bigcap_{F\in \mathcal{F}(\Delta^\alpha)}P_F=(x_{11}x_{12},x_2,\ldots, x_n)$ and the assertion holds. Let $\Delta=\langle F_1,\ldots,F_s\rangle$ and $x_1\in F_i$ for $i=1,\ldots,r$, and $x_1\not\in F_i$ for $i=r+1,\ldots,s$. Then
$$\begin{array}{rl}
  J_{\Delta^\alpha}= & [\bigcap_{F\in \mathcal{F}(\langle F_1\rangle^\alpha)}P_F]\cap \cdots\cap [\bigcap_{F\in \mathcal{F}(\langle F_s\rangle^\alpha)}P_F] \\
  = & [(x_{11}x_{12})+P_{F_1\backslash x_1}]\cap\cdots\cap[(x_{11}x_{12})+P_{F_r\backslash x_1}]\cap [\bigcap^s_{i=r+1}P_{F_i}]\\
  = & [(x_{11}x_{12})+\bigcap^r_{i=1}P_{F_i\backslash x_1}]\cap[\bigcap^s_{i=r+1}P_{F_i}]\\
  =& (x_{i_11}\cdots x_{i_t1}:x_{i_1}\cdots x_{i_t}\in J_\Delta, i_l\neq 1\ \mbox{for all}\ l)+\\
  & (x_{11}x_{12}x_{i_11}\cdots x_{i_{t'}1}:x_1x_{i_1}\cdots x_{i_{t'}}\in J_\Delta).
\end{array}$$

Suppose that $\alpha=(k_1,\ldots,k_n)\in \mathbb{N}^n$ is arbitrary with $k_1>1$ and let $\alpha=\beta+\delta_1$. By the induction hypothesis,
$$\begin{array}{rl}
  J_{\Delta^\beta}= & (\prod^{k_{i_1}}_{j=1}x_{i_1j}\cdots \prod^{k_{i_t}}_{j=1}x_{i_tj}:x_{i_1}\cdots x_{i_t}\in J_\Delta, i_l\neq 1\ \mbox{for all}\ l)+ \\
   &  (\prod^{k_1-1}_{j=1}x_{1j}\prod^{k_{i_1}}_{j=1}x_{i_1j}\cdots \prod^{k_{i_{t'}}}_{j=1}x_{i_{t'}j}:x_1x_{i_1}\cdots x_{i_{t'}}\in J_\Delta),
\end{array}$$
and so
$$\begin{array}{rl}
    J_{(\Delta^\beta)^{1+\delta_{1k_1}}}= & (\prod^{k_{i_1}}_{j=1}x_{i_1j1}\cdots \prod^{k_{i_t}}_{j=1}x_{i_tj1}:x_{i_1}\cdots x_{i_t}\in J_\Delta, i_l\neq 1\ \mbox{for all}\ l)+ \\
     &  (x_{1(k_1-1)2}\prod^{k_1-1}_{j=1}x_{1j1}\prod^{k_{i_1}}_{j=1}x_{i_1j1}\cdots \prod^{k_{i_{t'}}}_{j=1}x_{i_{t'}j1}:x_1x_{i_1}\cdots x_{i_{t'}}\in J_\Delta).
  \end{array}
$$
It follows from Lemma \ref{epsilon} that
$$J_{\Delta^\alpha}=(\{\prod^{k_{i_1}}_{j=1}x_{i_1j}\cdots \prod^{k_{i_r}}_{j=1}x_{i_rj}:x_{i_1}\cdots x_{i_r}\in J_{\Delta}\}).$$
\end{proof}

To prove Proposition \ref{facet}, we use the following Proposition.
\begin{proposition}\label{Hart}
(\cite[Proposition 1]{Ha}) Let $R$ be a Noetherian local ring containing a field $K$, and $u_1,\ldots,u_n$ be an $R$-sequence. Then the natural map
$$\begin{array}{cr}
    \varphi: & S=K[x_1,\ldots,x_n]\rightarrow R \\
     &x_i\mapsto u_i
  \end{array}
$$
of $K$-algebras is injective and $R$ is a flat $S$-module.
\end{proposition}

\begin{proposition}\label{facet}
Let $\Delta$ be a simplicial complex and $\alpha=(k_1,\ldots,k_n)\in \mathbb{N}^n$. Then
\begin{enumerate}[\upshape (i)]
  \item $\beta_i(S/J_\Delta)=\beta_i(S^\alpha/J_{\Delta^\alpha})$;
  \item $S/J_\Delta$ is Cohen-Macaulay if and only if $S^\alpha/J_{\Delta^\alpha}$ is Cohen-Macaulay.
\end{enumerate}
\end{proposition}
\begin{proof}
Define $\varphi:S\rightarrow S^\alpha$ given by $\varphi(x_i)=\prod^{k_i}_{j=1}x_{ij}$. Since $\prod^{k_1}_{j=1}x_{1j},\ldots,\prod^{k_n}_{j=1}x_{nj}$ is an $S^\alpha$-regular sequence, it follows from Proposition \ref{Hart} that $S^\alpha$ is a flat $S$-module. Now, by \cite[Theorem 2.1.7]{BrHe}, (ii) is concluded. Also, if $F_\bullet$ is a minimal free resolution of $S^\alpha/J_\Delta$ over $S$, then it follows that $F_\bullet\otimes_SS^\alpha$ is a minimal
free resolution of $S^\alpha/J_{\Delta^\alpha}$ over $S^\alpha$. Therefore we obtain (i).
\end{proof}

The following corollary shows that the regularity of the facet ideal does not change under the expansion functor.

\begin{corollary}\label{lres expan}
Let $\Delta$ be a simplicial complex and let $\alpha\in\mathbb{N}^n$. Then $\reg(I(\Delta))=\reg(I(\Delta^{\alpha}))$. Moreover, $I(\Delta)$ has a linear resolution if and only if $I(\Delta^\alpha)$ has a linear resolution.
\end{corollary}
\begin{proof}
It is a consequence of Proposition \ref{facet}, \cite[Theorem 2.1,Corollary 1.6]{Te}.
\end{proof}

\begin{remark}
Bayati and Herzog in \cite{BaHe} defined the expansion functor in the category of finitely generated
multigraded $S$-modules. They showed that a finitely generated multigraded $S$-module has a linear resolution if and only if its expansion does, too (c.f. \cite[Corollary 4.3]{BaHe}). A special case of their result is the second part of Corollary \ref{lres expan}. Then in \cite{RaYa}, the authors studied the expansion of monomial ideals in the concept of Bayati and Herzog. They showed that a monomial ideal has linear quotients if and only if its expansion does (c.f. \cite[Theorem 1.7]{RaYa}). Also, a monomial ideal is weakly polymatroidal if and only if its expansion is (c.f. \cite[Theorem 1.4]{RaYa}). As a consequence of these results we have:

If $\Delta$ is a simplicial complex on $[n]$ and $\alpha\in\mathbb{N}^n$, then
\begin{itemize}
  \item $I(\Delta)$ has linear quotients if and only if $I(\Delta^\alpha)$ has linear quotients;
  \item $I(\Delta)$ is weakly polymatroidal if and only if $I(\Delta^\alpha)$ is weakly polymatroidal.
\end{itemize}
\end{remark}

We use the following theorem to compare the projective dimension of a facet ideal $I(\Delta)$ with linear quotients with the projective dimension of $I(\Delta^\alpha)$.

\begin{theorem}\label{Leila}
(\cite[Corollary 2.7]{ShVa}) Let $I$ be a monomial ideal with linear quotients with the ordering $f_1<\cdots<f_m$ on the minimal generators of $I$.
Then $$\beta_{i,j}(I)=\sum_{\deg(f_t)=j-i} {|\set_I(f_t)|\choose i}.$$
\end{theorem}

\begin{proposition}\label{pdlinear}
If $I(\Delta)$ has linear quotients, $\alpha=(s,s,\ldots,s)$ and $d=\dim(\Delta)$, then $\pd(I(\Delta^\alpha))\leq \pd(I(\Delta))s+(d+1)(s-1)$. Moreover, if $\Delta$ is pure, then $$\pd(I(\Delta^\alpha))=\pd(I(\Delta))s+(d+1)(s-1).$$
\end{proposition}

\begin{proof}
Let $I(\Delta)$ has linear quotients. In view of the proof of \cite[Theorem 1.7]{RaYa}, consider an order on the minimal generators of $I(\Delta^{\alpha})$ as follows.

Fix an order of linear quotients for $I(\Delta)$. For two facets $F,F'\in \Delta^{\alpha}$, if $\overline{F}=\overline{F'}=\{x_{i_1},\ldots,x_{i_k}\}$, $F=\{x_{i_1r_1},\ldots,x_{i_kr_k}\}$ and $F'=\{x_{i_1r'_1},\ldots,x_{i_kr'_k}\}$, where $i_1<\cdots< i_k$, set $x^F<x^{F'}$ if and only if $(r_1,\ldots, r_{k})<_{lex} (r'_1,\ldots, r'_{k})$.
Otherwise set $x^F<x^{F'}$ if and only if $x^{\overline{F}}<x^{\overline{F'}}$ in the order of linear quotients for $I(\Delta)$.

$I(\Delta^\alpha)$ has linear quotients with respect to above order. In the light of Theorem \ref{Leila}, with this ordering, we have
$\pd(I(\Delta^{\alpha}))=\max\{|\set_{I(\Delta^{\alpha})}(x^F)|\ |\ F\in  \mathcal{F}(\Delta^\alpha)\}$.
One can see that for a minimal generator $x^F\in I(\Delta^{\alpha})$,
$$\set_{I(\Delta^\alpha)}(x^F)=\{x_{it_i}:\ \ x_i\in \set_{I(\Delta)}(x^{\overline{F}}),\ 1\leq t_i\leq s\}\cup \{x_{it_i}:\ x_{ir_i}\in F,\  t_i< r_i\}.$$  Thus for any $F\in \mathcal{F}(\Delta^{\alpha})$,
$$|\set_{I(\Delta^\alpha)}(x^F)|\leq |\set_{I(\Delta)}(x^{\overline{F}})|s+|F|(s-1)\leq \pd(I(\Delta))s+(d+1)(s-1).$$
Therefore
$\pd(I(\Delta^\alpha))\leq \pd(I(\Delta))s+(d+1)(s-1)$.

Now, assume that $\Delta$ is pure of dimension $d$ and let $\pd(I(\Delta))=|\set_{I(\Delta)}(x^{F})|$ for some $F\in \mathcal{F}(\Delta)$.
Let $F=\{x_{i_1},\ldots,x_{i_t}\}$. Then for $F'=\{x_{i_1s},\ldots,x_{i_ts}\}$ one has
$$|\set_{I(\Delta^\alpha)}(x^{F'})|=|\set_{I(\Delta)}(x^{F})|s+|F|(s-1)=\pd(I(\Delta))s+(d+1)(s-1),$$ noting the fact that
$\{x_{it_i}:\ x_i\in \set_{I(\Delta)}(x^{F}),\ 1\leq t_i\leq s\}\cap \{x_{it_i}| x_{ir_i}\in F',\  t_i< s\}=\emptyset$.
\end{proof}

\section{The expansion of graphs}
In this section, we consider the case when $\Delta=G$ is a graph and investigate some properties that are equivalent in $G$ and its expansion. We state conditions on a graph $G$ so that by adding a vertex to $G$ or removing a vertex from $G$, some properties of $G$ like Cohen-Macaulayness are preserved.

For a graph $G$ and a vertex $x\in V(G)$, let $N_G(x)=\{y\in V(G):\ \{x,y\}\in E(G)\}$ and $N_G[x]=N_G(x)\cup \{x\}$.

Let $G$ be a simple graph with the vertex set $X=\{x_1,\ldots,x_n\}$ and let $\alpha=(k_1,\ldots,k_n)\in\mathbb{N}^n$. The \emph{expansion of $G$ with respect to $\alpha$} is denoted by $G^\alpha$ and it is a simple graph with the vertex set $X^\alpha$ and the edge set
$$E(G^\alpha)=\{\{x_{ir},x_{js}\}:\ \ \{x_i,x_j\}\in E(G), 1\leq r\leq k_i,1\leq s\leq k_j\}.$$

In the following, the notion of a co-chordal (resp. co-shellable and co-Cohen-Macaulay) graph implies to a simple graph with chordal (resp. shellable and Cohen-Macaulay) complement.

\begin{theorem}\label{graph}
Let $G$ be a simple graph and $\alpha\in\mathbb{N}^n$.
\begin{enumerate}[\upshape (i)]
  \item $G$ is co-chordal if and only if $G^\alpha$ is;
  \item $G$ is co-shellable if and only if $G^\alpha$ is;
  \item $G$ is co-Cohen-Macaulay if and only if $G^\alpha$ is;
  \item $(\Delta_G)^{\vee}$ is vertex decomposable if and only if $(\Delta_{G^{\alpha}})^{\vee}$ is vertex decomposable.
\end{enumerate}
\end{theorem}
\begin{proof}

(i) In view of Corollary \ref{lres expan}, the edge ideal of a simple graph $G$ has a linear resolution if and only if the edge ideal of its expansion has a linear resolution. Combining this  with Fr\"{o}berg's result on edge ideals with a linear resolution (see \cite[Theorem 1]{Fro}), we get the assertion.

(ii), (iii) Considering the equalities $\Delta_{G^c}=\Delta(G)$ and $\Delta(G^\alpha)=\Delta(G)^\alpha$, the result follows from \cite[Theorem 2.4 , Corollary 2.15 ]{RaMo}.

(iv) follows from (i) and \cite[Corollary 3.8]{MK}.
\end{proof}

\begin{remark}\label{remexpan}
There is another notion for the expansion of a graph $G$ in the literature, which we denote it here by $\widehat{G^{\alpha}}$, to avoid the confusion with the above concept. For $\alpha=(k_1,\ldots,k_n)\in\mathbb{N}^n$,
$\widehat{G^\alpha}$  is a simple graph with the vertex set $X^\alpha$ and the edge set
$$E(\widehat{G^\alpha})=\{\{x_{ir},x_{js}\}:\ \{x_i,x_j\}\in E(G), 1\leq r\leq k_i,1\leq s\leq k_j\}\cup\{\{x_{ir},x_{is}\}:\ 1\leq i\leq n,\ r\neq s\}.$$
It is easy to see that $\Delta_{\widehat{G^\alpha}}=(\Delta_G)^{\alpha}$ (see for example \cite[Remark 2.4]{KhMo}).
\end{remark}

In view of the above remark, we conclude the following assertions.

\begin{corollary}\label{G1}
Let $G$ be a graph, $x\in V(G)$ and $G'$ be the graph obtained from $G$ by adding a new vertex $x'$ and connecting it to all vertices in $N_G[x]$. If $G$ is Cohen-Macaulay (resp. sequentially Cohen-Macaulay and shellable), then $G'$ is Cohen-Macaulay (resp. sequentially Cohen-Macaulay and shellable).
\end{corollary}
\begin{proof}
Note that $G'$ is an expansion of $G$ in the sense of Remark \ref{remexpan}. Thus $\Delta_{G'}$ is an expansion of $\Delta_G$ and hence \cite[Theorem 2.4, Corollaries 2.8, 2.15]{RaMo} imply the result.
\end{proof}
\begin{corollary}\label{G2}
Let $G$ be a graph and  $x,y\in V(G)$ be distinct vertices such that $N_G[x]=N_G[y]$. If $G$ is Cohen-Macaulay (resp. sequentially Cohen-Macaulay and shellable), then $G\setminus x$ is Cohen-Macaulay (resp. sequentially Cohen-Macaulay and shellable).
\end{corollary}

\begin{proof}
It is easy to see that $G$ is an expansion of $G\setminus x$ in the sense of Remark \ref{remexpan}. Thus $\Delta_G$ is an expansion of $\Delta_{G\setminus x}$.
\end{proof}

\section*{\bf Acknowledgments}
The research of the first author was in part supported by grants from IPM with number (No. 95130021). The second author was supported by the research council of the University of Maragheh and a grant from IPM (No. 94130029) .


\begin{thebibliography}{20}

\bibitem{BaHe}
S. Bayati and J. Herzog,
\newblock Expansions of monomial ideals and multigraded modules,
\newblock {\em To appear in Rocky Mountain J. Math.}

\bibitem{BiVa}
J. Biermann and A. Van Tuyl,
\newblock Balanced vertex decomposable simplicial complexes and their h-vectors,
\newblock {\em  arXiv:1202.0044v2 (2012).}

\bibitem{BrHe}
W. Bruns and J. Herzog,
\newblock {\em Cohen-Macaulay rings},
\newblock Revised edition. Cambridge University Press 1998.

\bibitem{CoNa}
D. Cook II and U. Nagel,
\newblock Cohen-Macaulay graphs and face vectors of flag complexes,
\newblock {\em SIAM J. Discrete
Math.}
\textbf{26} (2012), no. 1, 89-101.

\bibitem{FH}
C. A. Francisco and H. T. H\`{a},
\newblock Whiskers and sequentially Cohen-Macaulay graphs,
\newblock {\em J. Combin. Theory Ser. A}
\textbf{115} (2008), no. 2, 304-–316.

\bibitem{Fro}
R. Fr\"{o}berg,
\newblock On Stanley-Reisner rings,
\newblock {\em Topics in Algebra, Banach Center Publications}
\textbf{26} (1990),57-70.

\bibitem{Ha}
R. Hartshorne,
\newblock A property of A-sequences,
\newblock {\em Bull. de la S. M. F.}
\textbf{94} (1996), 61-65.

\bibitem{HeHi}
J. Herzog and T. Hibi,
\newblock {\em Monomial ideals},
\newblock Springer 2011.

\bibitem{KhMo}
S. Moradi and F. Khoshahang,
\newblock Expansion of  a simplicial complex,
\newblock {\em Journal of Algebra and its applications}
\textbf{15} (2016), no. 1, 1-–15.

\bibitem{MK}
S. Moradi and F. Khoshahang,
\newblock On vertex decomposable simplicial complexes and their Alexander duals,
\newblock {\em Math. Scand.}
\textbf{118} (2016), no. 1, 43-–56.

\bibitem{RaMo}
R. Rahmati-Asghar and S. Moradi,
\newblock On the Stanley-Reisner ideal of an expanded simplicial complex,
\newblock {\em Manuscripta Math.}
\textbf{150} (2016), no. 3-4, 533--545.

\bibitem{RY}
R. Rahmati-Asghar and S. Yassemi,
\newblock $k$-decomposable monomial ideals,
\newblock {\em Algebra Colloq.}
\textbf{22} (2015),  745-756.

\bibitem{RaYa}
R. Rahmati-Asghar and S. Yassemi,
\newblock The behaviors of expansion functor on monomial ideals and toric rings,
\newblock {\em Communications in Algebra}
\textbf{44} (2016), no. 9,  3874-3889.

\bibitem{ShVa}
L. Sharifan and M. Varbaro,
\newblock Graded Betti numbers of ideals with linear quotients,
\newblock {\em Le Matematiche (Catania)}
\textbf{63} (2008), no. 2,  257-265.

\bibitem{Te}
N. Terai,
\newblock Alexander duality theorem and Stanley-Reisner rings,
\newblock {\em Surikaisekikenkyusho Kokyu-ruko}
 (1999), no. 1078,  174-184,  Free resolutions of coordinate rings of projective varieties and related topics (Kyoto 1998).

\end{thebibliography}
\end{document}